\documentclass[onecolumn]{acmsiggraph}

\title{Distributions of missing sums and differences}

\author{Tiffany C. Inglis\\
University of British Columbia \& University of Waterloo\\
NSERC USRA Report \\
August 28, 2007\\
piffany@gmail.com}

\pdfauthor{Tiffany C. Inglis}

\usepackage{amsfonts}
\usepackage{amssymb}
\usepackage{amsthm}
\usepackage{subfig}
\usepackage{enumerate}
\usepackage[ruled, vlined]{algorithm2e}

\newcommand{\Prob}[1]{P\left[{#1}\right]}
\newcommand{\floor}[1]{{\left\lfloor{#1}\right\rfloor}}
\newcommand{\ceil}[1]{{\left\lceil{#1}\right\rceil}}
\newcommand{\bigO}[1]{{\mathcal{O}\left({#1}\right)}}
\newcommand{\abs}[1]{{\left|{#1}\right|}}

\newcommand{\Rinf}{R_{\infty}}

\theoremstyle{plain}
\newtheorem{thm}{Theorem}[section]
\newtheorem{lem}[thm]{Lemma}

\theoremstyle{definition}
\newtheorem{defn}{Definition}[section]
\newtheorem{conj}{Conjecture}[section]

\theoremstyle{remark}

\begin{document}




\maketitle












\section{Introduction}

For every set of integers $R = \{x_1,\ldots,x_r\}$, there is a corresponding sumset with all pairwise sums
\begin{eqnarray}
	R + R &:=& \{ x_i+x_j : 1 \leq i, j\leq r\}
\end{eqnarray}
and a difference set with all possible pairwise differences
\begin{eqnarray}
	R - R &:=& \{ x_i-x_j : 1 \leq i, j\leq r\}.
\end{eqnarray}
Various properties of these sets are of interest to us, in particular, when $R$ is taken to be a random subset of $\{0,\ldots,n-1\}$ for some positive integer $n$. Here we introduce some terminologies to be used throughout the paper.

\vspace{3mm}
\begin{defn}
For every $n \in \mathbb{N}$, let $R_n$ be the random variable denoting a uniformly randomly chosen subset of $\{0,\ldots,n-1\}$ with $\Prob{i \in R_n} = \frac{1}{2}$ for all $i$ satisfying $0 \leq i \leq n-1$. Define $R_n'$ and $R_n''$ similarly with the extra conditions that $0 \in R_n'$ and $\{0,n-1\} \subset R_n''$. A subscript of $\infty$ indicates that the set is taken from all nonnegative integers.
\end{defn}
\vspace{3mm}

\section{Probability of $n$ as a sum given that $n-1$ is a sum}

Associated with each integer $k \geq 0$ is the probability of it being found in $\Rinf + \Rinf$. Martin and O'Bryant~\shortcite{martin} derived an explicit form for this probability:
\begin{eqnarray}
	\Prob{k \in \Rinf + \Rinf} &=&
		\begin{cases}
			1 - \left(\frac{3}{4}\right)^{(k+1)/2}, & \mbox{if $k$ is odd} \\
			1 - \frac{1}{2} \left(\frac{3}{4}\right)^{k/2}, & \mbox{if $k$ is even.} \\
		\end{cases}
\label{eqn:1}
\end{eqnarray}

In the following theorem, we extend this notion to find the probability of $k$ being a sum given that $k-1$ is a sum.

\vspace{3mm}
\begin{thm}
For every integer $k \geq 1$, the conditional probability
\begin{eqnarray}
	\Prob{k \in \Rinf + \Rinf | k-1 \in \Rinf + \Rinf} &=&
	2 + \frac{F_{k+2} + (-1)^k 3^\floor{k/2} - 2^{k+1}}
					 {2\left(2^k - 3^{\floor{k/2}}\right)},
\end{eqnarray}
where $F_{k+2}$ is the $(k+2)$-th term in the Fibonacci sequence, givenly explicitly as
\begin{eqnarray}
	F_{k+2} &=& \frac{\left(1 + \sqrt{5}\right)^{k+2} - \left(1 - \sqrt{5}\right)^{k+2}}
									 {2^{k+2} \sqrt{5}}.
\end{eqnarray}
\label{thm:1}
\end{thm}
\vspace{3mm}

\begin{proof}
As the actual derivation is rather tedious, I will spare the details and simply outline the proof. By definition of conditional probability,
\begin{eqnarray}
	\Prob{k \in \Rinf + \Rinf | k-1 \in \Rinf + \Rinf} &=&
	\frac{\Prob{\{k-1,k\} \subset \Rinf + \Rinf}}
			 {\Prob{k-1 \in \Rinf + \Rinf}},
\label{eqn:2}
\end{eqnarray}
for which the denominator is easily determined. As for the numerator, let us introduce some new notations. For $i$ satisfying $0 \leq i \leq \frac{k}{2}$, let $E_i^k$ be the event that
\begin{eqnarray}
	\{i,k-i\} \subset \Rinf
	\vee
	\{i+1,k-(i+1)\} \subset \Rinf
	\vee \ldots \vee
	\{\floor{\tfrac{k}{2}},k-\floor{\tfrac{k}{2}}\} \subset \Rinf,
\label{eqn:case1}
\end{eqnarray}
and for $i$ satisfying $0 \leq i \leq \frac{k-1}{2}$, let $E_i^{k-1}$ be the event that
\begin{eqnarray}
	\{i,(k-1)-i\} \subset \Rinf
	\vee
	\{i+1,(k-1)-(i+1)\} \subset \Rinf
	\vee \ldots \vee
	\{\floor{\tfrac{k-1}{2}},(n-1)-\floor{\tfrac{k-1}{2}}\} \subset \Rinf.
\end{eqnarray}
Hence the numerator in Equation~\ref{eqn:2} may be re-expressed as
$\Prob{E_0^{k-1} \wedge E_0^k}$. It turns out that, with proper conditioning of terms, we can write down a recursive formula for probabilities of this form:
\begin{eqnarray}
	\Prob{E_i^{k-1} \wedge E_i^k} &=&
		\tfrac{1}{2} \Prob{E_{i+1}^{k-1} \wedge E_{i+1}^k} +
		\tfrac{1}{8} \Prob{E_{i+1}^{k-1} \wedge E_{i+2}^k} + \ldots \\
	\Prob{E_i^{k-1} \wedge E_{i+1}^k} &=&
		\tfrac{1}{2} \Prob{E_{i+1}^{k-1} \wedge E_{i+1}^k} +
		\tfrac{1}{4} \Prob{E_{i+1}^{k-1} \wedge E_{i+2}^k} + \ldots.
\end{eqnarray}
Now if we denote
$\Prob{E_i^{k-1} \wedge E_i^k}$
and 
$\Prob{E_i^{k-1} \wedge E_{i+1}^k}$
respectively by $P_{2i}(k)$ and $P_{2i+1}(k)$,
we have the recursive sequence
$\{P_i(k)\}_{0 \leq i \leq k-1}$
with $P_{k-1}(k) = P_{k-2}(k) = \tfrac{1}{4}$ and
\begin{eqnarray}
	P_i(k) &=& \tfrac{1}{2} P_{i+1}(k) +
						 \tfrac{1}{4} P_{i+2}(k) -
						 \tfrac{1}{24} \left(3+(-1)^{k-i}\right) \left(\tfrac{3}{4}\right)^\floor{(k-i)/2} +
						 \tfrac{1}{4}
						 \qquad \mbox{for $i=0,1,\ldots,k-3$}.
\end{eqnarray}
Noticing that the recursion depends on $k-i$ rather than on $k$ and $i$ independently prompts us to introduce yet another sequence
$T_j := P_{k-j}(k) = P_{k-j-i}(k-1) = \ldots = P_0(j)$
with $T_1 = T_2 = \frac{1}{4}$. For $k \geq 3$, the recursive relation works out to be
\begin{eqnarray}
	T_k &=& 1 + \frac{F_{k+2}}{2^{k+1}}
						- \frac{3^\floor{k/2}\left(4-(-1)^k\right)}
									 {2^{k+1}},
\end{eqnarray}
where $\{F_i\}$ is the Fibonacci sequence. From this formula yields the desired probability since
\begin{eqnarray}
	\Prob{k \in \Rinf + \Rinf | k-1 \in \Rinf + \Rinf}
		&=&	\frac{T_0}{\Prob{k-1 \in \Rinf + \Rinf}} \\
		&=& \frac{T_0}{1 - \tfrac{1}{4}\left(3 + (-1)^k\right)
												\left(\frac{3}{4}\right)^\floor{k/2}}.
\end{eqnarray}
\end{proof}
\section{Relating several distributions for number of missing sums}

Knowing the probability
$\Prob{k \in \Rinf + \Rinf | k-1 \in \Rinf + \Rinf}$
is only the tip of the iceberg. What we are really aiming to find out is, for any $n \geq 0$, how $R_n + R_n$ varies in size. Certainly, $\#\{R_n + R_n\}$ varies depending on $n$, so instead we turn our focus to the number of sums it misses in the range of all possible sums, $\{0,\ldots,2n-2\}$. For convenience, we will incorporate a few new notations. For any integer set $R$ and real interval $[a,b]$, define $f_{[a,b]}^+(R)$ as the number of sums $R+R$ misses in the interval $[a,b]$. Similarly, let $f_{[a,b]}^-(R)$ be the number of differences $R-R$ misses in $[a,b]$. More precisely,
\begin{eqnarray}
	f_{[a,b]}^{+}(R) &:=&
		\#\left\{ k \in \mathbb{Z} \cap [a,b] : k \notin R + R \right\},\\
	f_{[a,b]}^{-}(R) &:=&
		\#\left\{ k \in \mathbb{Z} \cap [a,b] : k \notin R - R \right\}.
\end{eqnarray}
For the three \emph{special} sets $R_n$, $R_n'$ and $R_n''$ defined earlier, we are primarily concerned with the sums they miss in the intervals $[0,n-1]$ and $[0,2n-2]$, along with the differences missed in $[0,n-1]$ and $[-(n-1),n-1]$. Note that the differences missed in $[-(n-1),n-1]$ are exactly those missed in $[0,n-1]$, plus their negative counterparts.

\begin{figure}[thb!]
  \centering
  \subfloat[]{\label{fig:1a}
  	\includegraphics[width=0.45\textwidth]
  	{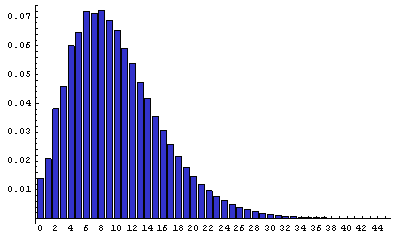}}\\
  \subfloat[]{\label{fig:1b}
  	\includegraphics[width=0.45\textwidth]
  	{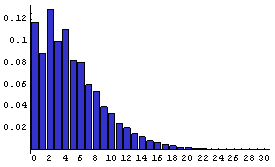}}
  \subfloat[]{\label{fig:1c}
  	\includegraphics[width=0.45\textwidth]
  	{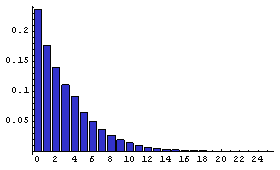}}
 	\caption{The distributions for the number of missing sums for (a) $R_n$ in $[0,2n-2]$, (b) $R_n$ in $[0,n-1]$, and (c) $R_n'$ in $[0,n-1]$.}
  	\label{fig:1}
\end{figure}

Martin and O'Bryant~\shortcite{martin} studied several of these probabilities, including
$\Prob{f_{[0,2n-2]}^{+}(R_n) = m}$,
the probability of $R_n + R_n$ missing exactly $m$ sums in $[0,2n-2]$. From randomly generated data, they speculated that underneath this distribution lies a more fundamental distribution involving 
$\Prob{f_{[0,n-1]}^{+}(R_n) = m}$,
which in turn is built upon the distribution for
$\Prob{f_{[0,n-1]}^{+}(R_n') = m}$ (Figure~\ref{fig:1}). The following theorem states the exact relationship between these distributions in the limiting case as $n$ approaches infinity:

\vspace{3mm}
\begin{thm}
Let $m$ be a nonnegative integer. Then we have
\begin{enumerate}[(a)]
	\item $\lim\limits_{n \to \infty} \Prob{f_{[0,2n-2]}^{+}(R_n) = m} =
				 \sum\limits_{i=0}^m{
				 		\lim\limits_{n \to \infty} \Prob{f_{[0,n-1]}^{+}(R_n) = i}
				 		\lim\limits_{n \to \infty} \Prob{f_{[0,n-1]}^{+}(R_n) = m-i}
				 		}$,
	\item $\lim\limits_{n \to \infty} \Prob{f_{[0,n-1]}^{+}(R_n) = m} =
				 \sum\limits_{i=0}^\floor{m/2}{
				 		\frac{1}{2^{i+1}}
				 		\lim\limits_{n \to \infty} \Prob{f_{[0,n-1]}^{+}(R_n') = m-2i}
				 		}$.
\end{enumerate}
\label{thm:2}
\end{thm}
\vspace{3mm}

The proof involves counting the missing sums separately, for example, splitting them into small sums (ranging from $0$ to $n-1$) and large sums (ranging from $n$ to $2n-2$). For written simplicity, we employ the following subscripts to denote the intervals involved (the possible parameters $R_n$ and $R_n'$ are omitted here):
$$
\begin{array}{lll}
		(a) f_{all}^+ := f_{[0,2n-2]}
	& (b) f_{S}^+   := f_{[0,n-1]}
	& (c) f_{L}^+   := f_{[n,2n-2]}   \\\\
  	(d) f_{XS}^+  := f_{[0,\floor{n/8}-1]}
  & (e) f_{MS}^+  := f_{[\floor{n/8},n-1]}
  & (f) f_{ML}^+  := f_{[n,(2n-2)-\floor{n/8}]}    \\\\
  	(g) f_{XL}^+  := f_{[(2n-2)-\floor{n/8}+1, 2n-2]}
  & (h) f_{M}^+   := f_{[\floor{n/8},(2n-2)-\floor{n/8}]}
\end{array}
$$

Here are a few lemmas to facilitate the proof of the theorem:

\vspace{5mm}
\begin{lem}
\begin{eqnarray}
	\lim\limits_{n \to \infty} \Prob{f_M^+(R_n) \geq 1} = 0.
\end{eqnarray}
\label{lem:3}
\end{lem}

\begin{proof}
	From Equation~\ref{eqn:1}, we can derive the inequality $\Prob{k \notin R_n + R_n} \leq \left(\tfrac{3}{4}\right)^\floor{k/2}$, so
	
\begin{eqnarray}
		\lim\limits_{n \to \infty}{\Prob{f_M^+(R_n) \geq 1}}
			&\leq&
				  \lim\limits_{n \to \infty} {
					\sum\limits_{k=\floor{n/8}}^{(2n-2)-\floor{n/8}}{
						\Prob{k \notin R_n + R_n}
					}
			} \\
			&\leq&
				  \lim\limits_{n \to \infty} {
					\sum\limits_{k=\floor{n/8}}^{(2n-2)-\floor{n/8}}{
						\left(\tfrac{3}{4}\right)^\floor{k/2}
					}
			} \\
			&\leq&
				  \lim\limits_{n \to \infty} {
					\sum\limits_{k=\floor{n/8}}^{(2n-2)-\floor{n/8}}{
						\left(\tfrac{3}{4}\right)^\floor{\floor{n/8}/2}
					}
			} \\
			&\leq&
				  \lim\limits_{n \to \infty} {
					\left(2n -2\floor{n/8}-1 \right)
					\left(\tfrac{3}{4}\right)^\floor{\floor{n/8}/2}
			} \\
			&=& 0,
\end{eqnarray}
since the exponential decay of
$\left(\tfrac{3}{4}\right)^\floor{\floor{n/8}/2}$
exceeds the linear growth of $2n - 2\floor{\frac{n}{8}}-1$.
A probability is always nonnegative, hence
$\lim\limits_{n \to \infty}{\Prob{f_M^+(R_n) \geq 1}}$
must equal zero.
\end{proof}

\vspace{5mm}
\begin{lem}
	For any integer $m \geq 0$, the limit of the probability
	\begin{eqnarray}
		\lim\limits_{n \to \infty} {
			\sum\limits_{
					a,c \geq 0, b \geq 1, a+b+c = m
					} {
				\Prob{
					f_{XS}^{+}(R_n) = a
					\wedge
					f_{M}^{+}(R_n) = b
					\wedge
					f_{XL}^{+}(R_n) = c
				}
			}
		}
		&=& 0.
	\end{eqnarray}
\label{lem:4}
\end{lem}

\begin{proof}
Using Lemma~\ref{lem:3}, we have
	\begin{eqnarray}
		\lim\limits_{n \to \infty} {
			\sum\limits_{
					a,c \geq 0, b \geq 1, a+b+c = m
					} {
				\Prob{
					f_{XS}^{+}(R_n) = a
					\wedge
					f_{M}^{+}(R_n) = b
					\wedge
					f_{XL}^{+}(R_n) = c
				}
			}
		}
		&\leq&
		\lim\limits_{n \to \infty}{
			\sum\limits_{b \geq 1} {
			\Prob{f_M^+(R_n) = b}
			}} \\
		&\leq&
		\lim\limits_{n \to \infty}{
			\Prob{f_M^+(R_n) \geq 1}
			} \\
		&=& 0.
	\end{eqnarray}
\end{proof}

\vspace{5mm}
\begin{lem}
For any integer $m$ satisfying $m \geq 0$, the probability
\begin{eqnarray}
	\Prob{f_{XS}^+(R_n) = i \wedge f_{XL}^{+}(R_n) = m-i}
	&=&
	\Prob{f_{XS}^+(R_n) = i}\Prob{f_{XL}^{+}(R_n) = m-i}.
\end{eqnarray}
\label{lem:5}
\end{lem}

\begin{proof}
To prove the lemma, we need to show that $f_{XS}^+(R_n) = i$ and $f_{XL}^{+}(R_n) = m-i$ are independent events. By definition,
\begin{eqnarray}
	f_{XS}^+(R_n) &:=& 
		f_{[0,\floor{n/8}-1]}(R_n) \\
		&=& \frac{\floor{\frac{n}{8}} - 
							\#\left\{\{0,\ldots,\floor{\frac{n}{8}}-1\} \cap (R_n + R_n) \right\}}
						 {\floor{\frac{n}{8}}} \\
		&=& \frac{\floor{\frac{n}{8}} - 
							\#\left\{
								\left(\{0,\ldots,\floor{\frac{n}{8}}-1\} \cap R_n \right) +
								\left(\{0,\ldots,\floor{\frac{n}{8}}-1\} \cap R_n \right)
								\right\}}
						 {\floor{\frac{n}{8}}},
\end{eqnarray}
where the last step uses the fact that any sum of $R_n$ in the range $\left[0,\floor{\frac{n}{8}}-1\right]$ can only be the sum of two elements in $R_n$ in that range. Similarly,
\begin{eqnarray}
	f_{XL}^+(R_n) &=&
		\frac{
					 \floor{\frac{n}{8}} -
					 \#\left\{
					 			\left(
					 				\left\{2n-\floor{\frac{n}{8}}-1,\ldots,2n-2\right\}
					 				\cap
					 				R_n
					 			\right) +
					 			\left(
					 				\left\{2n-\floor{\frac{n}{8}}-1,\ldots,2n-2\right\}
					 				\cap
					 				R_n
					 			\right)
					 	 \right\}
				 }
				 {\floor{\frac{n}{8}}},
\end{eqnarray}
since large sums only result from adding two large elements together.

Notice that $f_{XS}^+(R_n)$ is a function of $n$ and
$\left\{0,\ldots,\floor{\frac{n}{8}}-1\right\} \cap R_n$
while $f_{XL}^+(R_n)$
is a function of $n$ and
$\left\{2n-\floor{\frac{n}{8}}-1,\ldots,2n-2\right\} \cap R_n$.
Since
$\left\{0,\ldots,\floor{\frac{n}{8}}-1\right\} \cap R_n$
and
$\left\{2n-\floor{\frac{n}{8}}-1,\ldots,2n-2\right\} \cap R_n$
as disjoint, for any fixed $n$, the values of
$f_{XS}^+(R_n)$ and$ f_{XL}^+(R_n)$
are independent.
\end{proof}

\vspace{5mm}
\begin{lem}
For any integer $m$ satisfying $m \geq 0$, the limit of the probability
\begin{eqnarray}
	\lim\limits_{n \to \infty} {
		\sum\limits_{i=0}^{m} {
			\Prob{f_{XS}^+(R_n) = i}
			\Prob{f_{XL}^+(R_n) = m-i}
		}
	}
	&=&
	\lim\limits_{n \to \infty} {
		\sum\limits_{i=0}^{m} {
			\Prob{f_{S}^+(R_n) = i}
			\Prob{f_{L}^+(R_n) = m-i}
		}
	}.
\end{eqnarray}
\label{lem:6}
\end{lem}

\begin{proof}
To prove the statement, we will show both
\begin{eqnarray}
	\lim\limits_{n \to \infty} {
		\Prob{f_{XS}^+(R_n) = i}
	}
	&=&
	\lim\limits_{n \to \infty} {
		\Prob{f_{S}^+(R_n) = i}
	} \\
	\lim\limits_{n \to \infty} {
		\Prob{f_{XL}^+(R_n) = m-i}
	}
	&=&
	\lim\limits_{n \to \infty} {
		\Prob{f_{L}^+(R_n) = m-i}
	}.
\end{eqnarray}
Re-express the probability
\begin{eqnarray}	
\Prob{f_{S}^+(R_n) = i}
		&=& \Prob{f_{XS}^+(R_n) = i \wedge f_{MS}^+(R_n) = 0} +
				\Prob{f_{XS}^+(R_n) = i - f_M^+(R_n) \wedge f_{MS}^+(R_n) \geq 1} \\
		&=& \Prob{f_{XS}^+(R_n) = i \wedge f_{MS}^+(R_n) = 0} +
				\bigO{\Prob{f_{MS}^+(R_n) \geq 1}} \\
		&=& \Prob{f_{XS}^+(R_n) = i} -
				\Prob{f_{XS}^+(R_n) = i \wedge f_{MS}^+(R_n) \geq 1} +
				\bigO{\Prob{f_{MS}^+(R_n) \geq 1}} \\
		&=& \Prob{f_{XS}^+(R_n) = i} +
				\bigO{\Prob{f_{MS}^+(R_n) \geq 1}} +
				\bigO{\Prob{f_{MS}^+(R_n) \geq 1}} \\
		&=& \Prob{f_{XS}^+(R_n) = i} +
				\bigO{\Prob{f_{MS}^+(R_n) \geq 1}} \\
		&=& \Prob{f_{XS}^+(R_n) = i} +
				\bigO{\Prob{f_{M}^+(R_n) \geq 1}}.
\end{eqnarray}

Since $\lim\limits_{n \to \infty}{\Prob{f_M^+(R_n)} \geq 1} = 0$ by Lemma~\ref{lem:3}, we have
\begin{eqnarray}
\lim\limits_{n \to \infty}{\Prob{f_{S}^+(R_n) = i}} 
&=&
\lim\limits_{n \to \infty}{\Prob{f_{XS}^+(R_n) = i}}.
\end{eqnarray}
Analogously, we can show that
\begin{eqnarray}
\lim\limits_{n \to \infty}{\Prob{f_{L}^+(R_n) = i}} 
&=&
\lim\limits_{n \to \infty}{\Prob{f_{XL}^+(R_n) = i}}.
\end{eqnarray}
\end{proof}

Now we are in position to prove Theorem~\ref{thm:2}, starting with Part (a). To avoid redundancy, we will write $f^+$ in place of $f^+(R_n)$ throughout this proof, combined with various pre-defined subscripts. Now to analyze
$\Prob{f_{all}^+ = m}$, first partition $f_{all}^+$ into $f_{XS}^+$, $f_{M}^+$, and $f_{XL}^+$, followed by conditioning on whether $f_{M}^+ = 0$, as shown:
\begin{eqnarray}
\Prob{f_{all}^+ = m} 
				&=& \sum\limits_{a,b,c \geq 0, a+b+c = m}
						{\Prob{f_{XS}^+ = a
						 \wedge
						 f_{M}^+ = b
						 \wedge
						 f_{XL}^+ = c}
						 } \\
				&=& \sum\limits_{a,c \geq 0, a+c = m}
						{\Prob{f_{XS}^+ = a
						 \wedge
						 f_{M}^+ = 0
						 \wedge
						 f_{XL}^+ = c}
						 } +
						 \sum\limits_{a,c \geq 0, b \geq 1, a+b+c = m}
						{\Prob{f_{XS}^+ = a
						 \wedge
						 f_{M}^+ = b
						 \wedge
						 f_{XL}^+ = c}
						 } \\
				&=& \sum\limits_{i=0}^{m}
						{\Prob{f_{XS}^+ = i
						 \wedge
						 f_{M}^+ = 0
						 \wedge
						 f_{XL}^+ = m-i}
						 } +
						 \sum\limits_{a,c \geq 0, b \geq 1, a+b+c = m}
						{\Prob{f_{XS}^+ = a
						 \wedge
						 f_{M}^+ = b
						 \wedge
						 f_{XL}^+ = c}
						 }.
\end{eqnarray}
Taking the limit as $n$ approaches infinity then applying Lemma~\ref{lem:3}, we have
\begin{eqnarray}
\lim\limits_{n \to \infty}{\Prob{f_{all}^+ = m}} &=&
\lim\limits_{n \to \infty}{
		\sum\limits_{i=0}^{m} {
			\Prob{f_{XS}^+ = i \wedge f_M^+ = 0 \wedge f_{XL}^+ = m-i}
		}
	}.
\end{eqnarray}

The RHS of the equation can be rewritten as
\begin{eqnarray}
\lim\limits_{n \to \infty} {
	\sum\limits_{i=0}^m {
		\Prob{f_{XS}^+ = i \wedge f_{XL}^+ = m-i}
	}
} -
\lim\limits_{n \to \infty} {
	\sum\limits_{i=0}^m {
		\Prob{f_{XS}^+ = i
		\wedge
		f_{M}^+ \geq 1
		\wedge
		f_{XL}^+ = m-i}
	}
}.
\label{eqn:3}
\end{eqnarray}
Taking the absolute value of Equation~\ref{eqn:3}, we get
\begin{eqnarray}
&&
\abs{
	\lim\limits_{n \to \infty} {
		\sum\limits_{i=0}^m {
			\Prob{f_{XS}^+ = i \wedge f_{XL}^+ = m-i}
		}
	} -
	\lim\limits_{n \to \infty} {
		\sum\limits_{i=0}^m {
			\Prob{f_{XS}^+ = i \wedge f_M^+ = 0 \wedge f_{XL}^+ = m-i}
		}
	}
} \\
&\leq&
\lim\limits_{n \to \infty} {
		\sum\limits_{i=0}^m {
			\Prob{f_{XS}^+ = i \wedge f_M^+ \geq 1 \wedge f_{XL}^+ = m-i}
		}
	}
	\\
&\leq&
\lim\limits_{n \to \infty} {
			\Prob{f_M^+ \geq 1}
	} =  0,\mbox{ by Lemma~\ref{lem:3}.}
\end{eqnarray}
So the expression in Equation~\ref{eqn:3} equals zero. In other words,
\begin{eqnarray}
\lim\limits_{n \to \infty} {
	\sum\limits_{i=0}^m {
		\Prob{f_{XS}^+ = i \wedge f_M^+ = 0 \wedge f_{XL}^+ = m-i}
	}
}
&=&
\lim\limits_{n \to \infty} {
	\sum\limits_{i=0}^m {
		\Prob{f_{XS}^+ = i \wedge f_{XL}^+ = m-i}
	}
},
\end{eqnarray}
which implies
$
\lim\limits_{n \to \infty} {
	\Prob{f_{all}^+ = m}
}
=
\lim\limits_{n \to \infty} {
	\sum\limits_{i=0}^m{
	\Prob{f_{XS}^+ = i \wedge f_{XL}^+ = m-i}
}}.
$
Then by Lemma~\ref{lem:5}~and~\ref{lem:6},
\begin{eqnarray}
\lim\limits_{n \to \infty} {
	\Prob{f_{all}^+ = m}
}
&=&
\lim\limits_{n \to \infty} {
	\sum\limits_{i=0}^m {
		\Prob{f_{XS}^+ = i}\Prob{f_{XL}^+ = m-i}
	}
} \\
&=&
\lim\limits_{n \to \infty} {
	\sum\limits_{i=0}^m {
		\Prob{f_{S}^+ = i}\Prob{f_{L}^+ = m-i}
	}
}, \mbox{ by symmetry of $f_S$ and $f_L$}, \\
&=&
\lim\limits_{n \to \infty} {
	\sum\limits_{i=0}^m {
		\Prob{f_{S}^+ = i}\Prob{f_{S}^+ = m-i}
	}
}, \mbox{ as desired}.
\end{eqnarray}

Now for Part (b), begin by conditioning $\Prob{f_S^+(R_n) = m}$ on whether the sumset $R_n + R_n$ is empty, which yields
\begin{eqnarray}
\Prob{R_n + R_n = \emptyset}
\Prob{f_S^+(R_n) = m | R_n + R_n = \emptyset} +
\Prob{R_n + R_n \neq \emptyset}
\Prob{f_S^+(R_n) = m | R_n + R_n \neq \emptyset}.
\label{eqn:4}
\end{eqnarray}

The first term in Equation~\ref{eqn:4} equals
$
\Prob{R_n \neq \emptyset}
\Prob{f_S^+(R_n) = m | f_S^+(R_n) = n}
= \left(\frac{1}{2^n}\right) \cdot 1_{[m=n]}.
$
Hence
\begin{eqnarray}
\lim\limits_{n \to \infty} {
	\Prob{R_n + R_n = \emptyset}
	\Prob{f_S^+(R_n) = m | R_n + R_n = \emptyset}
}
&=& 0.
\end{eqnarray}

Then condition the second term in Equation~\ref{eqn:4} on the value of the minimum sum, $\min(R_n+R_n)$:
\begin{eqnarray}
&&\Prob{R_n + R_n \neq \emptyset}
\Prob{f_{[0,n-1]}^+(R_n) = m | R_n + R_n \neq \emptyset} \\
&=&
\sum\limits_{i=0}^{2n-2} {
	\Prob{\min(R_n+R_n) = i}
	\Prob{f_{[0,n-1]}^+(R_n) = m | \min(R_n+R_n) = i}
} \\
&=&
\sum\limits_{i=0}^{m} {
	\Prob{\min(R_n+R_n) = i}
	\Prob{f_{[0,n-1]}^+(R_n) = m | \min(R_n+R_n) = i}
} \\
&=&
\sum\limits_{i=0}^{\floor{m/2}} {
	\Prob{\min(R_n) = i}
	\Prob{f_{[0,n-1]}^+(R_n) = m | \min(R_n) = i}
} \\
&=&
\sum\limits_{i=0}^{\floor{m/2}} {
	\frac{1}{2^{i+1}}
	\Prob{f_{[0,n-1]}^+(R_n) = m | \min(R_n) = i}
} \\
&=&
\sum\limits_{i=0}^{\floor{m/2}} {
	\frac{1}{2^{i+1}}
	\Prob{f_{[0,n-1]}^+(R_{n-i}' + i) = m}
} \\
&=&
\sum\limits_{i=0}^{\floor{m/2}} {
	\frac{1}{2^{i+1}}
	\Prob{f_{[2i,n-1]}^+(R_{n-i}' + i) = m - 2i}
} \\
&=&
\sum\limits_{i=0}^{\floor{m/2}} {
	\frac{1}{2^{i+1}}
	\Prob{f_{[0,n-2i-1]}^+(R_{n-i}') = m - 2i}
}. \\
\end{eqnarray}
Express this result as
\begin{eqnarray}
&& \sum\limits_{i=0}^{\floor{m/2}} {
\frac{1}{2^{i+1}}
\Prob{f_{[0,n-i-1]}^+(R_{n-i}') - f_{[n-2i-1,n-i-1]}^+(R_{n-i}') = m-2i}
} \\
&=& \sum\limits_{i=0}^{\floor{m/2}}{
\frac{1}{2^{i+1}}
\Prob{f_{[0,n-i-1]}^+(R_{n-i}') = m-2i
			\wedge
			f_{[n-2i-1,n-i-1]}^+(R_{n-i}') = 0
			}
} - 
\\
&& - \sum\limits_{i=0}^{\floor{m/2}}{
\frac{1}{2^{i+1}}
\Prob{f_{[0,n-i-1]}^+(R_{n-i}') = m-2i
			\wedge
			f_{[n-2i-1,n-i-1]}^+(R_{n-i}') \geq 1
			}
}.
\label{eqn:5}
\end{eqnarray}

In Equation~\ref{eqn:5}, the limit of the second term
\begin{eqnarray}
&&
\lim\limits_{n \to \infty} {
\sum\limits_{i=0}^{\floor{m/2}} {
\frac{1}{2^{i+1}}
\Prob{f_{[0,n-i-1]}^+(R_{n-i}') = m-2i
\wedge
f_{[n-2i-1,n-i-1]}^+(R_{n-i}' \geq 1)
}}} \\
&\leq&
\lim\limits_{n \to \infty} {
\sum\limits_{i=0}^{\floor{m/2}} {
\frac{1}{2^{i+1}}
\Prob{f_{[n-2i-1,n-i-1]}^+(R_{n-i}') \geq 1}
}} \\
&\leq&
\lim\limits_{n \to \infty} {
\sum\limits_{i=0}^{\floor{m/2}} {
\frac{1}{2^{i+1}}
\Prob{f_{[n-2i-1,n-i-1]}^+(R_{n-i}) \geq 1}
}}, \mbox{ and for $n \gg m$},\\
&\leq&
\lim\limits_{n \to \infty} {
\sum\limits_{i=0}^{\floor{m/2}} {
\frac{1}{2^{i+1}}
\Prob{f_M^+(R_{n-i}) \geq 1}
}}=0, \mbox{ by Lemma~\ref{lem:3}}.
\end{eqnarray}

Hence,
$
\lim\limits_{n \to \infty} {
\sum\limits_{i=0}^{\floor{m/2}} {
\frac{1}{2^{i+1}}
\Prob{
f_{[0,n-i-1]}^+(R_{n-i}') = m-2i
\wedge
f_{[n-2i-1,n-i-1]}^+(R_{n-i}') \geq 1
}}} 
= 0.
$
We have just shown that
$\Prob{f_S^+(R_n) = m}$
expands into three terms, two of which converge to $0$ as $n \to \infty$, leaving us
\begin{eqnarray}
\lim\limits_{n \to \infty} {
	\Prob{f_S^+(R_n) = m}
}
&=&
\lim\limits_{n \to \infty} {
	\sum\limits_{i=0}^{\floor{m/2}} {
		\frac{1}{2^{i+1}}
		\Prob{
			f_{[0,n-i-1]}^+(R_{n-i}') = m - 2i
			\wedge
			f_{[n-2i-1,n-i-1]}^+(R_{n-i}') = 0
		}}} \\
&=&
\lim\limits_{n \to \infty} {
	\sum\limits_{i=0}^{\floor{m/2}} {
		\frac{1}{2^{i+1}}
		\Prob{
			f_{S}^+(R_{n-i}') = m - 2i
			\wedge
			f_{[n-2i-1,n-i-1]}^+(R_{n-i}') = 0
		}}}.
\end{eqnarray}

By the above equation, the difference
\begin{eqnarray}
&&\abs{
\lim\limits_{n \to \infty} {
\sum\limits_{i=0}^{\floor{m/2}} {
\frac{1}{2^{i+1}}
\Prob{f_S^+(R_{n-i}') = m-2i}
}}
-
\lim\limits_{n \to \infty} {
\Prob{f_S^+(R_n) = m}
}} \\
&\leq&
\lim\limits_{n \to \infty} {
\sum\limits_{i=0}^{\floor{m/2}} {
\frac{1}{2^{i+1}}
\Prob{f_S^+(R_{n-i}') = m-2i
		  \wedge
		  f_{[n-2i-1,n-i-1]}^+(R_{n-i}') \geq 1}
}} \\
&\leq&
\lim\limits_{n \to \infty} {
\sum\limits_{i=0}^{\floor{m/2}} {
\frac{1}{2^{i+1}}
\Prob{f_{[n-2i-1,n-i-1]}^+(R_{n-i}') \geq 1}
}} \\
&\leq&
\lim\limits_{n \to \infty} {
\sum\limits_{i=0}^{\floor{m/2}} {
\frac{1}{2^{i+1}}
\Prob{f_{[n-2i-1,n-i-1]}^+(R_{n-i}) \geq 1}
}}, \mbox{ and for $n \gg m$}, \\
&\leq&
\lim\limits_{n \to \infty} {
\sum\limits_{i=0}^{\floor{m/2}} {
\frac{1}{2^{i+1}}
\Prob{f_{M}^+(R_{n-i}) \geq 1}
}}, \mbox{ by Lemma~\ref{lem:3}}. \\
\end{eqnarray}

Hence,
\begin{eqnarray}
\lim\limits_{n \to \infty} {
\Prob{f_S^+(R_n) = m}
}
&=&
\lim\limits_{n \to \infty} {
\sum\limits_{i=0}^{\floor{m/2}} {
\frac{1}{2^{i+1}}
\Prob{f_S^+(R_{n-i}') = m-2i}
}} \\
&=&
\lim\limits_{n \to \infty} {
\sum\limits_{i=0}^{\floor{m/2}} {
\frac{1}{2^{i+1}}
\Prob{f_S^+(R_{n}') = m-2i}
}}, \mbox{ as required} .
\end{eqnarray}

In fact, with this method, we can extend Theorem~\ref{thm:2} as follows (proof omitted):

\vspace{3mm}
\begin{thm}
Let $m$ be a nonnegative integer. Then we have
\begin{enumerate} [(a)]
	\item $\lim\limits_{n \to \infty} {\Prob{f_1 = m}} = 
				 \sum\limits_{i=0}^{m}{
				 \lim\limits_{n \to \infty}{\Prob{f_2 = i}}
				 \lim\limits_{n \to \infty}{\Prob{f_2 = m-i}}
				 },
				 $\\
				 where $(f_1,f_2) \in
				 	\left\{
				 	  \left(
				 	    f_{[0,2n-2]}^+(R_n), f_{[0,n-1]}^+(R_n)
				 	  \right),
				 	  \left(
				 	    f_{[0,2n-2]}^+(R_n''), f_{[0,n-1]}^+(R_n')
				 	  \right)
				 	\right\}$
	\item $\lim\limits_{n \to \infty} {\Prob{f_3 = m}} = 
				 \sum\limits_{i=0}^{\floor{m/2}}{
				 \frac{1}{2^{i+1}}
				 \lim\limits_{n \to \infty}{\Prob{f_4 = m-2i}}
				 },
				 $\\
				 where $(f_3,f_4) \in
				 	\left\{
				 	  \left(f_{[0,n-1]}^+(R_n), f_{[0,n-1]}^+(R_n')\right),
				 	  \left(f_{[0,2n-2]}^+(R_n), f_{[0,2n-2]}^+(R_n')\right),
				 	  \left(f_{[0,2n-2]}^+(R_n'), f_{[0,2n-2]}^+(R_n'')\right),
				 	  \right.
				 $
				 \\ 
				 $
				 		\left.
				 		\hspace{24mm}
				 	  \left(f_{[0,2n-2]}^-(R_n), f_{[0,2n-2]}^-(R_n')\right),
				 	  \left(f_{[0,2n-2]}^-(R_n'), f_{[0,2n-2]}^-(R_n'')\right)
				 	\right\}$.
\end{enumerate}
\label{thm:2e}
\end{thm}
\vspace{3mm}

\section{Missing fewer sums and differences is more probable?}

Now that we know how the distributions of number of missing sums and differences are related, let us turn our attention to the properties of each one. The most fundamental distributions for sums and differences are $\Prob{f_{[0,n-1]}^+(R_n') = m}$ and $\Prob{f_{[0,n-1]}^-(R_n'') = m}$, shown in Figure~\ref{fig:2}.

\begin{figure}[thb!]
  \centering
  \subfloat[]{\label{fig:2a}
  	\includegraphics[width=0.45\textwidth]
  	{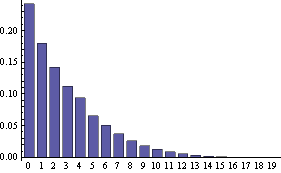}}\,\,
  \subfloat[]{\label{fig:2b}
  	\includegraphics[width=0.45\textwidth]
  	{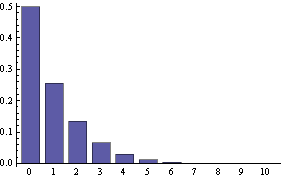}}
 	\caption{The observed distributions of (a) the number of missing sums for $R_n'$ in $[0,n-1]$ and (b) the number of missing differences for $R_n''$ in $[0,n-1]$.}
  	\label{fig:2}
\end{figure}
While both appear to be decreasing at a roughly exponential rate, the tiny distribution actually has a tiny \emph{blip} in the tail. We will state this as a theorem:

\vspace{3mm}
\begin{thm}
For $m \geq 0$, $\Prob{f_{[0,n-1]}^-(R_n'') = m}$ is not decreasing in $m$ except for $n=1,2,3,5,9$.
\end{thm}
\vspace{3mm}

\begin{proof}
Let $p_n(m) = \Prob{f_{[0,n-1]}^-(R_n'') = m}$. For $n \leq 10$, we can easily verify the claim by determining the exact function. Figure~\ref{fig:3} shows whether or not these functions are decreasing.

\begin{figure}[thb!]
  \centering
  \includegraphics[width=0.9\textwidth]
  {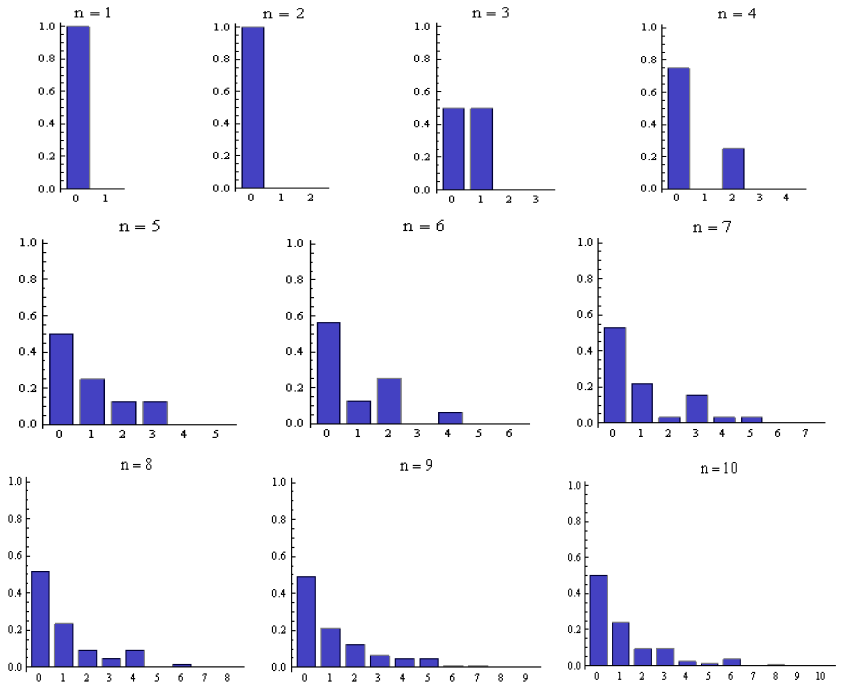}
  \caption{Exact distributions of the number of missing differences for small values of $n$}
  \label{fig:3}
\end{figure}

To prove the theorem for $n \geq 11$, we want to find some $m \in \{1,\ldots,n\}$ such that
\begin{eqnarray}
p_n(m) - p_n(m-1) > 0.
\end{eqnarray}

Let us evaluate $p_n(m)$ starting with $m=n, n-1, n-2, n-3, \ldots$.

\begin{itemize}
	\item $p_n(n) = \Prob{R_n'' - R_n'' = \emptyset} = 0$, since $\{0,n-1\} \subset R_n''$ implies $\{0,n-1\} \subset R_n'' - R_n''$
	\item $p_n(n-1) = 0$
	\item $p_n(n-2) = \Prob{R_n'' = \{0,n-1\}}
									= \frac{1}{2^{n-2}} $
	\item $p_n(n-3) = \Prob{R_n'' = \{0,\tfrac{n-1}{2}, n-1\}} 
									= \left(\frac{1}{2^{n-2}}\right) \cdot 1_{[2|n-1]}$
	\item 
	$
p_n(n-4) = \Prob{R_n'' = \{0,a,n-1\} | a \neq \tfrac{n-1}{2}} 
									+ \Prob{R_n'' = \{0,\tfrac{1}{3}(n-1),\tfrac{2}{3}(n-1),n-1\}}$ 	\\ 
								$
									\left. \hspace{14mm}
									 = \dfrac{n-2}{2^{n-2}} -
										\left(\dfrac{1}{2^{n-2}}\right) \cdot 1_{[2|n-1]} +
										\left(\dfrac{1}{2^{n-2}}\right) \cdot 1_{[3|n-1]}
									\right.
								$
	\item $
				p_n(n-5) =
					\Prob{R_n'' = \{0,d,(n-1)-d,n-1\} | d \neq \frac{1}{3}(n-1)} +
				$ \\
				$ \left. \hspace{17mm}
					+ \Prob{R_n'' = \{0,\frac{1}{4}(n-1),\frac{1}{2}(n-1),(n-1)\} |
								d_1 = d_2 = \frac{1}{2} d_3} +
				\right. $ \\
				$ \left. \hspace{17mm}
					+ \Prob{R_n'' = \{0,\frac{1}{2}(n-1),\frac{3}{4}(n-1),(n-1)\} |
								\frac{1}{2} d_1 = d_2 = d_3} +
				\right. $ \\
				$ \left. \hspace{17mm}
					+ \Prob{R_n'' = \{0,\frac{1}{4}(n-1),\frac{1}{2}(n-1),\frac{3}{4}(n-1),n-1\} } +
				\right. $ \\
				$ \left. \hspace{15mm}
					= \dfrac{\ceil{\frac{n-1}{2}}-1}{2^{n-2}}
					 - \left(\dfrac{1}{2^{n-2}}\right) \cdot 1_{[3|n-1]}
					 + \left(\dfrac{3}{2^{n-2}}\right) \cdot 1_{[4|n-1]}
				\right. . $ \\
\end{itemize}

So in the case of $m = n-4$,
\begin{eqnarray}
p_n(m) - p_n(m-1)
	&=& p_n(n-4) - p_n(n-5) \\
	&=& \frac{n-2}{2^{n-2}} - 
			\left(\tfrac{1}{2^{n-2}}\right) 1_{[2|n-1]} +
			\left(\tfrac{1}{2^{n-2}}\right) 1_{[3|n-1]} -
			\frac{\ceil{\tfrac{n-1}{2}}-1}{2^{n-2}} +
			\left(\tfrac{1}{2^{n-2}}\right) 1_{[3|n-1]} -
			\left(\tfrac{3}{2^{n-2}}\right) 1_{[4|n-1]} \\
	&\geq&
			\frac{n-2}{2^{n-2}} - \frac{1}{2^{n-2}} -
			\frac{\ceil{\tfrac{n-1}{2}}-1}{2^{n-2}} -
			\frac{3}{2^{n-2}} \\
	&\geq&
			\frac{n-5-\tfrac{n}{2}}{2^{n-2}} \\
	&=& \frac{n-10}{2^{n-1}} > 0.
\end{eqnarray}

\end{proof}

Using the same idea of analyzing the tail-end distribution, we can extend Theorem~\ref{thm:7} to include several distributions that seem to be decreasing after a certain point (usually the maximum) but in fact are not due to small blips in the tails.

\vspace{3mm}
\begin{thm}
In the following distributions, the function reaches a maximum before exhibiting a decreasing trend. For each distribution, let $m^*$ be where the maximum is reached,then we have
\begin{enumerate} [(a)]
\item For $m \geq m^*$,
			$\Prob{f_{[0,n-1]}^-(R_n'') = m}$
			is not decreasing except for $n=1,2,3,5,9$.
\item For $m \geq m^*$,
			$\Prob{f_{[0,n-1]}^-(R_n') = m}$
			is not decreasing except for $n=1,2$.
\item For $m \geq m^*$,
			$\Prob{f_{[0,n-1]}^-(R_n) = m}$
			is not decreasing except for $n=1$.
\item For $m \geq m^*$,
			$\Prob{f_{[0,2n-2]}^+(R_n'') = m}$
			is not decreasing except for $n=1,2$.
\item For $m \geq m^*$,
			$\Prob{f_{[0,2n-2]}^+(R_n') = m}$
			is not decreasing except for $n=1$.
\item For $m \geq m^*$,
			$\Prob{f_{[0,2n-2]}^+(R_n) = m}$
			is not decreasing except for $n=1$.
\end{enumerate}
For large enough values of $n$, these distributions stabilize and reach a limit, as shown in Figure~\ref{fig:4}.
\label{thm:7}
\end{thm}
\vspace{3mm}

\begin{figure}[thb!]
  \centering
  \subfloat[]{\label{fig:4a}
  	\includegraphics[width=0.32\textwidth]
  	{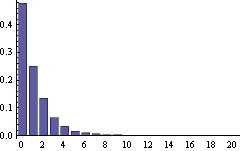}}\,\,
  \subfloat[]{\label{fig:4b}
  	\includegraphics[width=0.32\textwidth]
  	{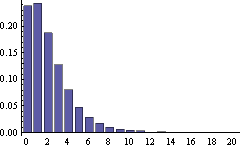}}\,\,
  \subfloat[]{\label{fig:4c}
  	\includegraphics[width=0.32\textwidth]
  	{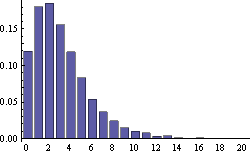}} \\
    \subfloat[]{\label{fig:4d}
  	\includegraphics[width=0.32\textwidth]
  	{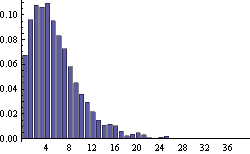}}\,\,
  \subfloat[]{\label{fig:4e}
  	\includegraphics[width=0.32\textwidth]
  	{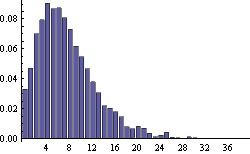}}\,\,
  \subfloat[]{\label{fig:4f}
  	\includegraphics[width=0.32\textwidth]
  	{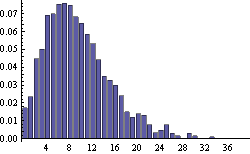}}
 	\caption{Non-decreasing distributions corresponding to Theorem~\ref{thm:7}}
  	\label{fig:4}
\end{figure}

Basically, the theorem says that, even in the parts of the graphs that exhibit a downward trend, these functions are not strictly decreasing due to some blips in the tail-end. Notice that the theorem does not assert anything for the limiting distributions because it is very possible for the limiting function to be decreasing, as in the case of $\Prob{f_{[0,n-1]}^-(R_n'')=m}$, for example. Another interesting observation is that the functions dealing with the number of missing sums in the interval $[0,n-1]$ may be decreasing. We will formally state this as a conjecture:

\vspace{3mm}
\begin{conj}
For each distribution, let $m^*$ be where the maximum is reached, then we have
\begin{enumerate} [(a)]
\item For $m \geq m^*$, $\Prob{f_{[0,n-1]}^+(R_n'')=m}$ is decreasing.
\item For $m \geq m^*$, $\Prob{f_{[0,n-1]}^+(R_n')=m}$ is decreasing except for $n=4,5,6$.
\item For $m \geq m^*$, $\Prob{f_{[0,n-1]}^+(R_n)=m}$ is decreasing except for $n=1$.
\end{enumerate}
\label{conj:8}
\end{conj}
\vspace{3mm}

These claims are supported by Figure~\ref{fig:5}.

\begin{figure}[thb!]
  \centering
  \subfloat[]{\label{fig:5a}
  	\includegraphics[width=0.32\textwidth]
  	{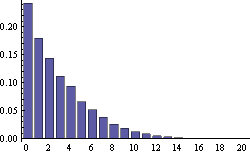}}\,\,
  \subfloat[]{\label{fig:5b}
  	\includegraphics[width=0.32\textwidth]
  	{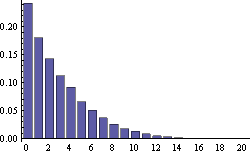}}\,\,
  \subfloat[]{\label{fig:5c}
  	\includegraphics[width=0.32\textwidth]
  	{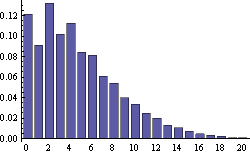}}
 	\caption{Distributions that appear to be decreasing as described in Conjecture~\ref{conj:8}}
  	\label{fig:5}
\end{figure}
\section{A conjecture about the limiting distribution for differences}

Recall that
$\Prob{f_{[0,n-1]}^-(R_{\infty}'') = m}$
from Figure~\ref{fig:4a} and 
$\Prob{f_{[0,n-1]}^+(R_{\infty}') = m}$
from Figure~\ref{fig:5b} are the two fundamental distributions, but despite our efforts, we still do not know much about either except some decreasing properties from the last section. Presently, we want to find out more about the heght of the bars in the graphs through experimentation in Mathematica.

Let us start with counting the sums. Instead of looking at all of $R_{\infty}'$, which almost surely spans $[0,\infty)$, we will focus on only the elements of $R_{\infty}'$ the interval $[0,19]$. By definition, $0 \in R_{\infty}'$. As for the other elements in $[1,19]$, we will experiment with all combinations. Since for each experiment, we fix which elements from $[0,19]$ are in $R_{\infty}'$, we also know exactly which elements from $[0,19]$ are in $R_{\infty}' + R_{\infty}'$ because elements greater than 19 do not affect sums in between 0 and 19. Counting the missing sums for these sets in the interval $[0,19]$, we have a distribution as shown in Figure~\ref{fig:6a}.

Notice the striking similarity between this graph and Figure~\ref{fig:5a}. This resemblance signifies that very few sets miss more than 20 sums, and the sums missed are almost less than 20. In fact, we can observe this pattern by focusing on an event smaller interval, say $[0,13]$. Very few missed sums exceed 13.

\begin{figure}[thb!]
  \centering
  \subfloat[]{\label{fig:6a}
  	\includegraphics[width=0.45\textwidth]
  	{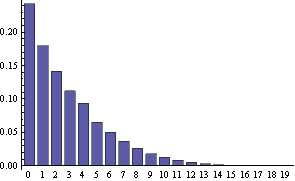}}\,\,
  \subfloat[]{\label{fig:6b}
  	\includegraphics[width=0.45\textwidth]
  	{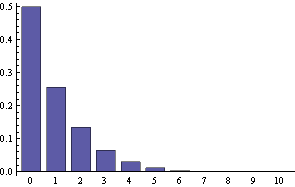}}
 	\caption{Distributions for (a) the number of missing sums in $[0,19]$, and (b) the number of missing differences in $[0,10]$}
  	\label{fig:6}
\end{figure}

Likewise, we can take a finite set, insist that both endpoints be included, and iterate through all possible combinations. Counting all relevant pairwise differences, we arrive at a distribution as shown in Figure~\ref{fig:6b}, which bears an uncanny resemblance to Figure~\ref{fig:4a}, hence demonstrating that nearly all missed differences are among the 10 largest possible differences. So if $d_{max}$ is the largest possible difference, then $\{d_{max}-9,\ldots,d_{max}\}$ would comprise virtually all the missed differences.

Let us draw our attention to the graph concerning missing differences. Recall that Figure~\ref{fig:6b} is obtained by iterating through all subsets of $[0,19]$ containing $\{0,19\}$ and counting the number of missing differences from $[0,10]$. Now, for $n=1,\ldots,9$, we can do the same by taking all subsets of $[0,2n-1]$ containing $\{0,2n-1\}$ and counting missing differences from $[0,n]$, we get the following distributions as shown in Figure~\ref{fig:7}.

\begin{figure}[thb!]
  \centering
  \subfloat[$n=1$]{\label{fig:7a}
  	\includegraphics[width=0.32\textwidth]
  	{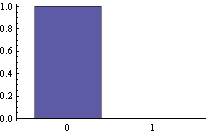}} \,\,
  \subfloat[$n=2$]{\label{fig:7b}
  	\includegraphics[width=0.32\textwidth]
  	{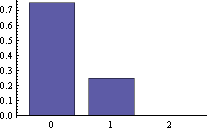}} \,\,
  \subfloat[$n=3$]{\label{fig:7c}
  	\includegraphics[width=0.32\textwidth]
  	{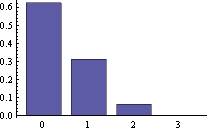}} \\
  \subfloat[$n=4$]{\label{fig:7d}
  	\includegraphics[width=0.32\textwidth]
  	{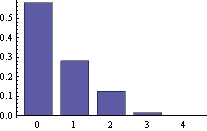}} \,\,
  \subfloat[$n=5$]{\label{fig:7e}
  	\includegraphics[width=0.32\textwidth]
  	{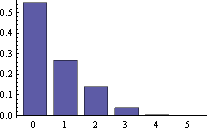}} \,\,
  \subfloat[$n=6$]{\label{fig:7f}
  	\includegraphics[width=0.32\textwidth]
  	{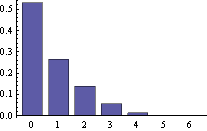}} \\
  \subfloat[$n=7$]{\label{fig:7g}
  	\includegraphics[width=0.32\textwidth]
  	{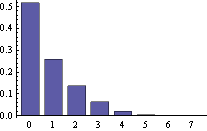}} \,\,
  \subfloat[$n=8$]{\label{fig:7h}
  	\includegraphics[width=0.32\textwidth]
  	{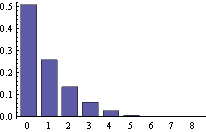}} \,\,
  \subfloat[$n=9$]{\label{fig:7i}
  	\includegraphics[width=0.32\textwidth]
  	{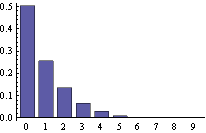}}  	
  \caption{Distribution of the number of missing differences for each value of $n$}
  	\label{fig:7}
\end{figure}

These graphs suggest that as $n$ approaches $\infty$, a limiting distribution exists and, in addition, is the same as the one shown in Figure~\ref{fig:4a}. Let us state this as a lemma and prove it before proceeding.

\vspace{3mm}
\begin{lem}
\begin{eqnarray}
\lim\limits_{n \to \infty} {
\Prob{f_{[0,n-1]}^-(R_n'')=m}
}
&=&
\lim\limits_{n \to \infty} {
\Prob{f_{[n,2n-1]}^-(R_{2n}'')=m}
}.
\end{eqnarray}
\end{lem}

\begin{proof}
We have two known identities. The first
\begin{eqnarray}
\lim\limits_{n \to \infty} {
\Prob{f_{[0,n-1]}^-(R_n'')=m}
}
&=&
\lim\limits_{n \to \infty} {
\Prob{f_{[0,2n-1]}^-(R_{2n}'')=m}
}.
\end{eqnarray}
is obtained by substituting $n$ with $2n-1$, and the second is
\begin{eqnarray}
\lim\limits_{n \to \infty} {
\Prob{f_{[0,2n-1]}^-(R_{2n}'')=m}
}
&=&
\sum\limits_{i=0}^{m}{\left(
\lim\limits_{n \to \infty} {
\Prob{f_{[0,n-1]}^-(R_{2n}'')=i}
}
\lim\limits_{n \to \infty} {
\Prob{f_{[n,2n-1]}^-(R_{2n}'')=m-i}
}
\right)}.
\end{eqnarray}

So to prove the lemma, we need to show that
$
\lim\limits_{n \to \infty} {
\Prob{f_{[0,n-1]}^-(R_{2n}'')=0}
} = 1
$.
Taking the probability of its complement, we have
\begin{eqnarray}
\lim\limits_{n \to \infty} {
\Prob{f_{[0,n-1]}^-(R_{2n}'') \geq 1}
}
&\leq&
\lim\limits_{n \to \infty} {
\sum\limits_{k=0}^{n}{
\Prob{k \notin R_{2n}'' - R_{2n}''}
}} \\
&=& \lim\limits_{n \to \infty} {
\sum\limits_{k=0}^{n}{
\left(\tfrac{3}{4}\right)^{2n-k}
}}\\
&\leq& \lim\limits_{n \to \infty} {
(n+1)\left(\tfrac{3}{4}\right)^n
} = 0.
\end{eqnarray}

Then since
$
\lim\limits_{n \to \infty} {
\Prob{f_{[0,n-1]}^-(R_{2n}'') \geq 1}
} = 0
$, we conclude that
\begin{eqnarray}
\lim\limits_{n \to \infty} {
\Prob{f_{[0,n-1]}^-(R_{n}'') = m}
}
&=&
\lim\limits_{n \to \infty} {
\Prob{f_{[0,2n-1]}^-(R_{2n}'') = m}
} \\
&=&
\lim\limits_{n \to \infty} {
\Prob{f_{[n,2n-1]}^-(R_{2n}'') = m}
}.
\end{eqnarray}

\end{proof}

The significance of all this is that, while
$
\Prob{f_{[0,n-1]}^-(R_n'') = m}
$
is, in general, not a decreasing function,
$
\Prob{f_{[0,n-1]}^-(R_{2n}'') = m}
$
\emph{does} seem to be decreasing, according to Theorem~\ref{thm:7}(a). If the decreasing trend continues, then its limiting distribution would also be decreasing. Note that this does not contradict with Theorem~\ref{thm:7}(a) since a sequence of functions does not need to be decreasing for its limit to be decreasing. In the proof of the theorem, we showed that a small blip at the tail-end of the distribution prevents it from being decreasing. If the blip is the only anomaly in an otherwise decreasing function, then as $n$ approaches $\infty$, the height of the blip would decrease to $0$, resulting in a decreasing limiting distribution. Figure~\ref{fig:7} serves to strengthen this last conjecture:

\vspace{3mm}
\begin{conj}
For all integers $m \geq 0$,
$
\lim\limits_{n \to \infty} {
\Prob{f_{[0,n-1]}^-(R_n'')=m}
}
$
is a decreasing function.
\label{conj:10}
\end{conj}
\vspace{3mm}

\bibliographystyle{acmsiggraph}
\bibliography{references}
\end{document}